%% file: ex_article.tex
\documentclass[hidelinks,onefignum,onetabnum]{siamart220329}

\input{ex_shared}

\ifpdf
\hypersetup{
  pdftitle={High-order proximal point algorithm  for the monotone variational inequality problem and its application},
  pdfauthor={Jingyu Gao and Xiurui Geng}
}
\fi

\begin{document}

\maketitle

\begin{abstract}
  The proximal point algorithm (PPA) has been developed to solve the monotone variational inequality problem. It provides a theoretical foundation for some methods, such as the augmented Lagrangian method (ALM) and the alternating direction method of multipliers (ADMM). This paper generalizes the PPA to the $p$th-order ($p\geq 1$) and proves its convergence rate $O \left(1/k^{p/2}\right)$ . Additionally, the $p$th-order ALM is proposed based on the $p$th-order PPA. Some numerical experiments are presented to demonstrate the performance of the $p$th-order ALM.
\end{abstract}



\section{Introduction}
This paper is inspired by the recent work of Nesterov \cite{nesterov2021inexact2,nesterov2021inexact} where the high-order proximal point method is proposed to minimize a $p$ times continuously differentiable and convex function $f$. The $p$th-order proximal-point operator is defined as \cite{nesterov2021inexact2}
\begin{equation}\label{new1}
  \mathsf{P}\mathrm{rox}_{\lambda f}^{p}\left(x\right)  : = \arg \min \left\{\lambda f\left(y\right) +\frac{1}{p+1} \left\lVert y-x\right\rVert ^{p+1}\vert \; y\in \mathbb{R} ^{n}\right\}, \quad \lambda >0.   
\end{equation}
Consider the following optimization problem
\begin{equation}\label{new2}
  \min \left\{f\left(x\right)| \; x\in \mathbb{R} ^{n} \right\},
\end{equation}
where $ f:\mathbb{R} ^{n} \to \mathbb{R} $ is a $p$ times continuously differentiable closed convex function. To solve the above problem, the $p$-order proximal point algorithm (PPA) can be expressed as 
\begin{equation}\label{new3}
    x^{k+1} = \mathsf{P}\mathrm{rox}_{\lambda f}^{p}\left(x^{k}\right), \quad k\geq 0,
\end{equation}
where $x^{k}$ is the approximate solution obtained after $k$ iterations. Its convergence rate  is that $f(x^{k})-\min f \leq O (1/k^{p})$. Moreover, the accelerated version of the $p$-order PPA for \cref{new2} converges as  $O (1/k^{p+1})$. Indeed, the optimization problem in \cref{new2} is equivalent to find a solution $x^{\ast }$ in $\mathbb{R} ^{n}$, which satisfies the following inequality
\begin{equation}\label{new4}
  (x-x^{\ast })^{T}\nabla f(x^{\ast })\geq 0, \quad \forall x \in \mathbb{R} ^{n}.
\end{equation}
In this paper, we consider the more general problem: the monotone variational inequality (VI) problem. The monotone VI problem in $\mathbb{R} ^{n}$ is to find $x^{\ast } \in \Omega $ such that 
\begin{equation}\label{new5}
  \mathrm{VI}(\Omega,F)  \qquad \qquad (x-x^{\ast })^{T}F(x^{\ast })\geq 0, \quad \forall x \in \Omega,
\end{equation}
where $\Omega$ is a nonempty closed convex set in $\mathbb{R} ^{n}$  and $F$ is a monotone mapping. The monotone VI problem plays a central role in nonlinear analysis, and it  gives a general description of some mathematical problems including optimization problems, complementarity problems, and finding Nash equilibria. The PPA is a fundamental method to  solve $\mathrm{VI}(\Omega,F)$ \cite{gol1979modified, nemirovski2004prox}, and it is the root of a number of famous methods such as the augmented Lagrangian method (ALM) \cite{hestenes1969multiplier, powell1969method}, the alternating direction method of multipliers (ADMM) \cite{boyd2011distributed, hong2017linear}, the Douglas-Rachford
operator splitting method (DRSM) \cite{lions1979splitting, combettes2007douglas}, and so on. 

The iterative scheme of the PPA for $\mathrm{VI}(\Omega,F)$ is given by
\begin{equation}\label{new6}
   x^{k+1}\in \Omega, \quad (x-x^{k+1})^{T}{[\lambda F(x^{k+1})+(x^{k+1}-x^{k})]}\geq 0, \quad \forall x\in \Omega,
\end{equation}
where $\lambda$ is the proximal parameter. As discussed in \cite{nemirovski2004prox}, the PPA is  convergent with the $O (1/ k)$ rate, where $k$ denotes the iteration number. Note that the methods originated from the original PPA such as ALM and ADMM have the same convergence rate $O (1/k)$ \cite{he20121, he2015non}. 

In this work, to accelerate the original PPA for $\mathrm{VI}(\Omega,F)$, we replace  $(x^{k+1}-x^{k})$ in \cref{new6} by $\left\lVert x^{k+1}-x^{k}\right\rVert^{p-1} (x^{k+1}-x^{k})$ with $p\geq 1$, and the corresponding method is named the $p$th-order PPA. Additionally, the convergence rate of the $p$th-order PPA is given. Based on the $p$th-order PPA, we present the $p$th-order ALM for the linearly constrained convex optimization problem, and prove that the $p$th-order ALM can be justified by using the $p$th-order PPA in the dual space.


The paper is organized as follows. In \cref{sec2}, we introduce some basic  notations and properties. In \cref{PPA}, we prove the 
convergence rate of the $p$th-order PPA. In \cref{sec4}, we develop the $p$th-order ALM based the $p$th-order PPA.  In \cref{sec6}, we conduct some numerical experiments to demonstrate the performance of the $p$th-order ALM with  some different options for  $p$.

\section{Preliminaries and notation}\label{sec2}
Given $x = \left[x_1, x_2,\ldots ,x_n\right]^{T} \in \mathbb{R} ^{n}$, we denote the $\ell_{2}$ norm and $\ell_{1}$ norm as
\begin{equation}\label{notation1}
   \left\lVert x\right\rVert = \sqrt{\sum_{i = 1}^{n} x_i^{2} }, \qquad \left\lVert x\right\rVert _1 = \sum_{i = 1}^{n}  \left\lvert x_i\right\rvert, \nonumber
\end{equation}
respectively. In the following, we make extensive use of the conjugate of a convex function. Given a convex function $f$, its conjugate function $f^{\ast }$ is defined as 
\begin{equation}\label{notation2}
  f^{\ast }(y)=\mathop {\sup }\limits_x  x^{T}y-f(x). \nonumber
\end{equation}
For example, when $f(x) = I_{\left\{b\right\}} (x)$, where $I_{\left\{b\right\}} (x)$ is the characteristic function of a single point set $ \left\{b\right\} $, it holds that
\begin{equation}\label{notation3}
  f^{\ast }(y)=b^{T}y. 
\end{equation}
The following lemma gives an important property of the conjugate function \cite{boyd2004convex}.
\begin{lemma}\label{newlemma_3}
  Suppose $f\left(x\right) $ is a closed proper convex function, and $f^{\ast }\left(y\right) $ is the corresponding conjugate function. Then,
  \begin{equation}\label{new45}
      y\in \partial f\left(x\right) \Leftrightarrow x \in \partial f^{\ast }\left(y\right).
  \end{equation}
\end{lemma}

Finally, for the sake of the notation, we denote
\begin{equation}\label{notation4}
  i_p(x) = \left\{ {\begin{array}{*{20}{r}}
    \frac{x}{\left\lVert x\right\rVert ^{1-\frac{1}{p} }}\qquad x\neq 0\\
    0  \quad  \qquad x=0
    \end{array}} \right. .
\end{equation}
Note that $i_p(x)$ is the gradient of the convex function $\frac{1}{1+\frac{1}{p} } \left\lVert x\right\rVert ^{1+\frac{1}{p}}$.

\section{Convergence analysis of the pth-order PPA}\label{PPA}
The $p$th-order PPA to solve $\mathrm{VI}(\Omega,F)$ is listed in \cref{algo1}.
\begin{algorithm} 
  \caption{$p$th-order PPA}
  \label{algo1}
  \begin{algorithmic}[1]
  \STATE{Require: $u_0 \in \mathbb{R}^{n},\; \lambda >0, \; p\geq 1$}
  \FOR{$k = 0,1,2,\ldots $}
  \STATE\label{line1}{Find $x^{k+1} \in \Omega$ satisfies that 
  \begin{equation} \label{new7}
     (x-x^{k+1})^{T}{[\lambda F(x^{k+1})+\left\lVert x^{k+1}-x^{k}\right\rVert^{p-1} (x^{k+1}-x^{k})]}\geq 0, \quad \forall x\in \Omega.
  \end{equation}
  }
  \ENDFOR
  \end{algorithmic}
\end{algorithm}

We assume the solution set $\Omega^{\ast }$ is nonempty.  For $\mathrm{VI}(\Omega,F)$, we can define $\widetilde{x} $ as an $\varepsilon $-approximate solution $(\varepsilon >0)$ which satisfies 
\begin{equation}\label{new8}
  \widetilde{x} \in \Omega, \qquad \qquad (x-\widetilde{x})^{T}F(\widetilde{x})\geq -\varepsilon, \qquad \forall x \in B\left(\widetilde{x}\right),
\end{equation}
where $B\left(\widetilde{x}\right)=\left\{x \in \Omega | \left\lVert x-\widetilde{x}\right\rVert \leq 1\right\}  $.
The goal of this section is to establish the global rate of convergence of the $p$th-order PPA for $\mathrm{VI}(\Omega,F)$, and we will see that the $p$th-order PPA converges much faster.
\begin{lemma}\label{newlemma_1}
  Let the sequence $\left\{x^{k} \right\}_{k\geq  0} $ be generated by \cref{algo1}, and $ x^{\ast } \in \Omega^{\ast } $. Then, for any $k\geq 0 $ we have
  \begin{equation}\label{new9}
      \left\lVert x ^{k+1}-x ^{\ast }\right\rVert ^{2}\leq       \left\lVert x ^{k}-x ^{\ast }\right\rVert ^{2}-      \left\lVert x ^{k+1}-x ^{k }\right\rVert ^{2}.
  \end{equation}  
\end{lemma}

\begin{proof}
  Setting $x$ in \cref{new7} as $x^{\ast }$, we have
  \begin{equation}\label{new10}
     \left\lVert x^{k+1}-x^{k}\right\rVert ^{p-1}(x^{\ast }-x^{k+1})^{T}(x^{k+1}-x^{k})\geq \lambda (x^{k+1}-x^{\ast })^{T}F(x^{k+1}).
  \end{equation}  
  Since $F$ is monotone, it holds that 
  \begin{equation}\label{new11}
    (x^{k+1}-x^{\ast })F(x^{k+1}) \geq (x^{k+1}-x^{\ast })^{T}F(x^{\ast}).
  \end{equation} 
  Combining \cref{new10}, \cref{new11}, and \cref{new5}, we get
  \begin{equation}\label{new12}
    \left\lVert x^{k+1}-x^{k}\right\rVert ^{p-1}(x^{\ast }-x^{k+1})^{T}(x^{k+1}-x^{k})\geq 0.
 \end{equation} 
 Note that the  identity
 \begin{equation}\label{new13}
     a^{T}b=\frac{1}{2} \left(\left\lVert a\right\rVert ^{2}+\left\lVert b\right\rVert ^{2}-\left\lVert a-b\right\rVert ^{2}\right) 
 \end{equation}
 holds for all $a$ and $b$. Setting $a=x ^{\ast }-x ^{k+1}$ and $b=x ^{k }-x ^{k+1}$ in \cref{new13}, and then combining with \cref{new12}, we finally obtain \cref{new9}. 
\end{proof}

\cref{newlemma_1} guarantees that the sequence $\{ \left\lVert x ^{k}-x ^{\ast }\right\rVert\} $ monotonically decreases. Thus, similar to the original PPA ($p=1$), the $p$th-order PPA ($p>1$) is also a contraction method. In the following lemma, we prove that $\{ \left\lVert x ^{k+1}-x ^{k }\right\rVert\} $ is also monotonically decreasing.

\begin{lemma}\label{newlemma_2}
  Let the sequence $\left\{x^{k} \right\}_{k\geq  0} $ be generated by \cref{algo1}. Then, for any $k\geq 1 $ we have
  \begin{equation}\label{new14}
    \left\lVert x ^{k+1}-x ^{k}\right\rVert \leq     \left\lVert x ^{k}-x ^{k-1}\right\rVert.
  \end{equation}
\end{lemma}

\begin{proof}
  Setting $x$ in \cref{new7} as $x^{k }$, we have
  \begin{equation}\label{new15}
    \left\lVert x^{k+1}-x^{k}\right\rVert ^{p-1}(x^{k }-x^{k+1})^{T}(x^{k+1}-x^{k})\geq \lambda (x^{k+1}-x^{k })^{T}F(x^{k+1}).
 \end{equation}  
  Replacing $x^{k+1}$ by $x^{k}$ in \cref{new7} and setting $x$ in the new variation inequality as $x^{k+1 }$, we obtain 
  \begin{equation}\label{new16}
    \left\lVert x^{k}-x^{k-1}\right\rVert ^{p-1}(x^{k }-x^{k-1})^{T}(x^{k+1}-x^{k})\geq \lambda (x^{k}-x^{k+1 })^{T}F(x^{k}).
 \end{equation} 
  Since $F$ is monotone, it holds that 
  \begin{equation}\label{new17}
    (x^{k+1}-x^{k })^{T}(F(x^{k+1})-F(x^{k})) \geq 0.
  \end{equation} 
  Putting the inequalities \cref{new15}, \cref{new16} and \cref{new17} together, we obtain
  \begin{equation}\label{new18}
    \left\lVert x^{k}-x^{k-1}\right\rVert ^{p-1}(x^{k }-x^{k-1})^{T}(x^{k+1}-x^{k})\geq \left\lVert x^{k+1}-x^{k}\right\rVert ^{p+1}.
  \end{equation}
  Note that $\left(x ^{k+1}-x ^{k}\right) ^{T} \left(x ^{k}-x ^{k-1}\right)\leq \left\lVert x ^{k+1}-x ^{k}\right\rVert \left\lVert x ^{k}-x ^{k-1}\right\rVert $, we have
  \begin{equation}\label{new19}
    \left\lVert x ^{k+1}-x ^{k}\right\rVert ^{p+1}\leq \left\lVert x ^{k}-x ^{k-1}\right\rVert ^{p} \left\lVert x ^{k+1}-x ^{k}\right\rVert,
 \end{equation}
  which implies \cref{new14} holds.
\end{proof}

Now, we are ready to prove the $ O \left(1/k^{p/2}\right) $ convergence rate of the $p$th-order PPA.

\begin{theorem}\label{newTheorem1}
  Let the sequence $\left\{x^{k} \right\}_{k\geq  0} $ be generated by \cref{algo1}. Then, for any $k\geq 0 $ we have
  \begin{equation}\label{new20}
    (x-x^{k+1})^{T}F(x^{k+1})
     \geq -\frac{1}{\lambda } \frac{\left\lVert x ^{0}-x ^{\ast }\right\rVert ^{p}}{\left(k+1\right) ^{p/2}} , \qquad \forall x \in B\left(x^{k+1}\right),
  \end{equation}
  where $B\left(x^{k+1}\right)=\left\{x \in \Omega | \left\lVert x-x^{k+1}\right\rVert \leq 1\right\}  $.
\end{theorem}

\begin{proof}
  First of all, we analysis the convergence of $\left\lVert x ^{k}-x ^{k+1}\right\rVert $. 

  According to \cref{newlemma_1}, we have
  \begin{equation}\label{new21}
   \begin{aligned}
   \sum\limits_{t = 0}^{k} \left\lVert x ^{t}-x ^{t+1}\right\rVert ^{2} &\leq \sum\limits_{t = 0}^{k} \left(\left\lVert x ^{t}-x ^{\ast }\right\rVert ^{2}-\left\lVert x ^{t+1}-x ^{\ast }\right\rVert ^{2}\right) \\
   & = \left\lVert x ^{0}-x ^{\ast }\right\rVert ^{2}-\left\lVert x ^{k+1}-x ^{\ast }\right\rVert ^{2}\\
   & \leq \left\lVert x ^{0}-x ^{\ast }\right\rVert ^{2}
   \end{aligned}
  \end{equation}
  In view of \cref{newlemma_2}, it holds that
  \begin{equation}\label{new22}
    (k+1)\left\lVert x ^{k}-x ^{k+1}\right\rVert^{2} \leq \sum\limits_{t = 0}^{k} \left\lVert x ^{t}-x ^{t+1}\right\rVert ^{2}.
  \end{equation}
  Combining \cref{new22} and \cref{new21}, we get
  \begin{equation}\label{add1}
    \left\lVert x ^{k}-x ^{k+1}\right\rVert^{2} \leq \frac{\left\lVert x^{0}-x^{\ast}\right\rVert^{2} }{k+1} .
  \end{equation}

  To prove the assertion \cref{new20}, it remains to note that according to \cref{new7}, we have
  \begin{equation}\label{new23}
   \begin{aligned}
    (x-x^{k+1})^{T}F(x^{k+1})&\geq \frac{1}{\lambda } \left\lVert x ^{k}-x ^{k+1}\right\rVert ^{p-1}\left(x -x ^{k+1}\right) ^{T}\left(x ^{k}-x ^{k+1}\right)\\ &\geq -\frac{1}{\lambda } \left\lVert x ^{k}-x ^{k+1}\right\rVert ^{p-1}\left\lVert x -x ^{k+1}\right\rVert \left\lVert x ^{k}-x ^{k+1}\right\rVert \\
    & = -\frac{1}{\lambda } \left\lVert x ^{k}-x ^{k+1}\right\rVert ^{p}\left\lVert x -x ^{k+1}\right\rVert \\ &\geq -\frac{1}{\lambda }\frac{\left\lVert x ^{0}-x ^{\ast }\right\rVert ^{p}}{\left(k+1\right) ^{p/2}} \left\lVert x -x ^{k+1}\right\rVert, \forall x \in \Omega.
   \end{aligned}
  \end{equation} 
\end{proof}

\begin{remark}
  \cref{newTheorem1} indicates that the $k$th iteration point  generated by the $p$th-order PPA is an approximate solution of  $\mathrm{VI}(\Omega,F)$ with an accuracy of $ O \left(1/k^{p/2}\right) $.  
  
\end{remark}
  
\section{Application to the linearly constrained convex
optimization problem} \label{sec4}
In this section, we reformulate the  linearly constrained convex
optimization problem into  a monotone VIP, and utilize the similar technique used in $p$th-order PPA to design the $p$th-order augmented Lagrangian method ($p$th-order ALM). 
\subsection{pth-order ALM}
Consider the following convex minimization problem with linear equality
constraints
\begin{equation}\label{new24}
   \min \left\{f\left(x\right)| Ax=b, x\in \mathcal{X}  \right\},
\end{equation}
where $ f:\mathbb{R} ^{n} \to \mathbb{R} $ is a closed proper convex but not necessarily smooth function,
$ \mathcal{X} \subseteq \mathbb{R} ^{n} $ is a
closed convex set, $ A \in \mathbb{R} ^{m\times n}$, and $ b \in \mathbb{R} ^{m}$. The Lagrangian function of the problem \cref{new24}
is
\begin{equation}\label{new25}
    L\left(x,\lambda \right) = f\left(x\right) +\lambda ^{T }\left(A x-b\right) ,
\end{equation}
with $\left(x,\lambda \right) \in \mathcal{X} \times \mathbb{R} ^{m}$. The pair $\left(x^{\ast } ,\lambda^{\ast } \right)$ is called a saddle point of the Lagrangian function $L\left(x,\lambda \right)$, if 
\begin{equation}\label{new26}
  L\left(x^{\ast },\lambda \right)\leq L\left(x^{\ast },\lambda^{\ast} \right) \leq L\left(x,\lambda^{\ast} \right), \quad \forall \left(x,\lambda \right) \in \mathcal{X} \times \mathbb{R} ^{m}.
\end{equation}
An equivalent expression of the saddle point is the following variational inequality:
\begin{equation}\label{new27}
  \left\{ {\begin{array}{*{20}{r}}
    {x^{\ast } \in \mathcal{X}, \qquad  \left(x-x^{\ast }\right) ^{T }  \left(\partial f(x^{\ast })+A ^{T }\lambda ^{\ast }\right)  \geq 0, \qquad \forall x \in\mathcal{X}}\\
    {\left(\lambda -\lambda ^{\ast }\right) ^{T}\left(-Ax^{\ast }+b\right)\geq 0, \quad \forall \lambda \in \mathbb{R} ^{m} }.
    \end{array}} \right.
\end{equation}

By denoting 
\begin{equation}\label{new28}
    w = \left( {\begin{array}{*{20}{c}}
    x\\
    \lambda 
    \end{array}} \right),     F\left(w\right)  = \left( {\begin{array}{*{20}{c}}
      \partial f(x)+A^{T}\lambda \\
      -Ax+b 
      \end{array}} \right), \text{and } \Omega =\mathcal{X} \times \mathbb{R} ^{m},
\end{equation}
\cref{new27} can be transformed into the compact form 
\begin{equation}\label{new29}
   w^{\ast } \in \Omega, \qquad \left(w-w^{\ast }\right) ^{T } F\left(w^{\ast }\right)   \geq 0, \quad \forall w \in \Omega.
\end{equation}
It can be easily verified that $F$ is a monotone mapping. Therefore, the optimization problem \cref{new24} is equivalent to a monotone VI problem. 

The ALM is a fundamental method to solve \cref{new24}. The iterative scheme of ALM for \cref{new24}
can be expressed as 
\begin{equation}\label{new30}
 \text{ALM} \left\{ {\begin{array}{*{20}{l}}
    {x^{k+1}\in \arg \min  \left\{f\left(x\right) + ( \lambda ^{k} ) ^T (A x-b)+\frac{\beta }{2}\left\lVert A x-b\right\rVert ^{2} \vert  x\in \mathcal{X} \right\}, }\\
    {\lambda ^{k+1} = \lambda ^{k}+\beta \left(A x^{k+1}-b\right) },
    \end{array}} \right.
\end{equation}
where the positive parameter $ \beta $ is known as the penalty parameter. The variation inequality form of \cref{new30} is
\begin{equation}\label{new31}
 \left\{ {\begin{array}{*{20}{r}}
    {x^{k+1 } \in \mathcal{X}, \qquad  \left(x-x^{k+1 }\right) ^{T }  \left(\partial f(x^{k+1})+A ^{T }\lambda ^{k+1 }\right)  \geq 0, \qquad \forall x \in\mathcal{X}}\\
    {\left(\lambda -\lambda ^{k+1 }\right) ^{T}\left(-Ax^{k+1 }+b+\frac{1}{\beta }(\lambda ^{k+1}-\lambda ^{k}) \right)\geq 0, \quad \forall \lambda \in \mathbb{R} ^{m} }.
    \end{array}} \right.
\end{equation}

Inspired by the $p$th-order PPA, we replace $(\lambda ^{k+1}-\lambda ^{k})$ in \cref{new31} by $\left\lVert \lambda ^{k+1}-\lambda ^{k}\right\rVert ^{p-1}\\(\lambda ^{k+1}-\lambda ^{k})$. The new iterative scheme is given by
\begin{equation}\label{new32}
\left\{ {\begin{array}{*{20}{r}}
    {x^{k+1 } \in \mathcal{X}, \quad  \left(x-x^{k+1 }\right) ^{T }  \left(\partial f(x^{k+1})+A ^{T }\lambda ^{k+1 }\right)  \geq 0, \quad \forall x \in\mathcal{X}}\\
    {\left(\lambda -\lambda ^{k+1 }\right) ^{T}\left(-Ax^{k+1 }+b+\frac{1}{\beta }\left\lVert \lambda ^{k+1}-\lambda ^{k}\right\rVert^{p-1}(\lambda ^{k+1}-\lambda ^{k}) \right)\geq 0,  \forall \lambda \in \mathbb{R} ^{m} }.
    \end{array}} \right.
\end{equation}
The new iterative scheme \cref{new32} is named as the $p$th-order ALM. Using the notation \cref{new28}, \cref{new32} has a compact form
\begin{equation}\label{new33}
  (w-w^{k+1})^{T}{ F(w^{k+1})+\frac{1}{\beta } \left\lVert \lambda ^{k+1}-\lambda ^{k}\right\rVert^{p-1} (\lambda -\lambda ^{k+1})^{T}(\lambda ^{k+1}-\lambda ^{k})}\geq 0, \quad \forall w\in \Omega.
\end{equation}
With the similar proof in \cref{PPA}, we can also obtain the convergence rate for the $p$th-order ALM, which is the same as the $p$th-order PPA. In the following theorem, we give the equivalent form of \cref{new32}.

\begin{theorem}\label{newTheorem2}
  Let the sequence $\left\{w^{k} = (x^{k}, \lambda ^{k})\right\}_{k\geq  0} $ be generated by the following iteration
  \begin{equation}\label{new34}
   \left\{ {\begin{array}{*{20}{l}}
    {x^{k+1} \in \arg \min  \left\{f\left(x\right) + ( \lambda ^{k} ) ^T (A x-b)+\frac{\beta^{\frac{1}{p} } }{1+\frac{1}{p} }\left\lVert A x-b\right\rVert ^{1+\frac{1}{p} } \vert x \in  \mathcal{X} \right\}}\\
    \lambda ^{k+1} =  \lambda ^{k}+\beta^{\frac{1}{p} } i_p(Ax^{k+1}-b),
    \end{array}} \right.
  \end{equation}
  where $i_p(x)$ is defined in \cref{notation4}. Then, for any $k\geq 0 $, $(x^{k}, \lambda ^{k})$ satisfies \cref{new32}
\end{theorem}

\begin{proof}
  For $x^{k+1}$, due to the first-order optimality condition, we have 
  \begin{equation}\label{new35}
    (x-x^{k+1})^{T}[\partial f(x^{k+1})+A^{T}\lambda ^{k}+\beta^{\frac{1}{p} } A^{T} i_p(Ax^{k+1}-b) ]\geq 0, \quad \forall x\in \mathcal{X} .
  \end{equation}
  Note that the update for $\lambda ^{k+1}$ in \cref{new34} is
  \begin{equation}\label{new36}
  \lambda ^{k+1} = \lambda ^{k}+\beta^{\frac{1}{p} }  i_p(Ax^{k+1}-b).
  \end{equation}
  Then, \cref{new35} can be rewritten as 
  \begin{equation}\label{new37}
    \left(x-x^{k+1}\right) ^{T}\left(\partial f(x^{k+1})+A^{T}\lambda ^{k+1}\right) \geq 0, \quad  \forall x \in \mathcal{X}.
  \end{equation}  
  From \cref{new36}, we can easily get
  \begin{equation}\label{new38}
      \left\lVert Ax^{k+1}-b\right\rVert=\frac{1}{\beta } \left\lVert \lambda ^{k}-\lambda ^{k+1}\right\rVert ^{p}. \nonumber
  \end{equation}  
  Moreover, it can be verified that
  \begin{equation}\label{new39}
      -Ax^{k+1}+b+\frac{1}{\beta } \left\lVert \lambda ^{k}-\lambda ^{k+1}\right\rVert ^{p-1}\left(\lambda ^{k+1}-\lambda ^{k}\right) =0.
  \end{equation}   
  \cref{new37} and \cref{new39} imply that  $(x^{k}, \lambda ^{k})$ satisfies \cref{new32}.
\end{proof}

\subsection{pth-order ALM from the dual perspective}
When $p=1$, the original ALM can be justified by using the PPA in the dual space \cite{rockafellar1976augmented}. Here, we generalize this basic conclusion to the $p$th-order case.

Consider the more general optimization problem
\begin{equation}\label{new40}
     \min \left\{f\left(x\right) +h\left(A x\right) | x\in \mathbb{R} ^{n}\right\} ,
\end{equation}
where $f\left(\cdot \right) $ and $h\left(\cdot \right) $ are the closed proper convex functions. Note that \cref{new40} is equivalent to  the  convex minimization problem with linear equality \cref{new24} when
$h\left(\cdot \right) $ is the characteristic function of a single point set $ \left\{b\right\} $. For \cref{new40}, we introduce an additional variable $y$, and get the equivalent problem 
\begin{equation}\label{new41}
  \min \left\{f\left(x\right) +h\left(y\right) | A x=y,x\in \mathbb{R} ^{n}\right\} .
\end{equation}
The $p$th-order ALM for \cref{new41} is as follows
\begin{equation}\label{new42}
  \left\{ {\begin{array}{*{20}{l}}
    {\left(x^{k+1},y^{k+1}\right) \in \arg \min\limits_{x,y}    \left\{f\left(x\right)+h\left(y\right) + ( \lambda ^{k} ) ^T (A x-y)+\frac{\beta^{\frac{1}{p} } }{1+\frac{1}{p} }\left\lVert A x-y\right\rVert ^{1+\frac{1}{p} }\right\}, }\\
    {\lambda ^{k+1} = \lambda ^{k}+\beta^{\frac{1}{p} }  i_p(Ax^{k+1}-y^{k+1}) }.
    \end{array}} \right.
\end{equation}

The dual problem of \cref{new40} is 
\begin{equation}\label{new43}
   \max \limits_z -f^{\ast }\left(-A^T z\right) -h^{\ast }\left(z\right) ,
\end{equation}
where $f^{\ast }\left(\cdot \right) $ and $h^{\ast }\left(\cdot \right) $ are the conjugate function of $f\left(\cdot \right) $ and $h\left(\cdot \right) $,   respectively. The $p$th-order proximal point method  \cite{nesterov2021inexact2} for the dual problem \cref{new43} is
\begin{equation}\label{new44}
  z ^{k+1} = \arg \min \limits_z \left\{f^{\ast }\left(-A^T z\right) +h^{\ast }\left(z\right)+\frac{1}{\beta} \frac{1}{p+1} \left\lVert z-z^{k}\right\rVert^{p+1} \right\} .
\end{equation}
In the following theorem, we will show the equivalent relationship between \cref{new42} and \cref{new44}.
\begin{theorem}\label{newTheorem3}
  Suppose $f\left(x\right) $ and $g\left(y\right) $ are the closed proper convex functions. Let the sequence $\left\{\left(x^{k},y^{k},\lambda ^{k}\right) \right\}_{k\geq  0} $ be generated by the high-order ALM \cref{new42}. Then, for any $k\geq 0 $ we have 
  \begin{equation}\label{new46}
      \lambda ^{k+1} \in \arg \min \limits_u \left\{f^{\ast }\left(-A^T u\right) +h^{\ast }\left(u\right)+\frac{1}{\beta} \frac{1}{p+1} \left\lVert u-\lambda ^{k}\right\rVert^{p+1} \right\}.
  \end{equation}
  where $f^{\ast }\left(\cdot \right) $ and $h^{\ast }\left(\cdot \right) $ are the conjugate function of $f\left(\cdot \right) $ and $h\left(\cdot \right) $,   respectively.
\end{theorem}  

\begin{proof}
  In the $p$th-order ALM \cref{new42}, the update for $ \left(x^{k+1},y^{k+1}\right) $ is equivalent to the following optimization problem
  \begin{equation}\label{new47}
     \min \left\{f\left(x\right)+h\left(y\right)+\left(\lambda^{k}\right) ^{T }z+\beta^{\frac{1}{p} } \frac{1}{1+\frac{1}{p} }\left\lVert z\right\rVert ^{1+\frac{1}{p} } \vert \; A x-y=z\right\}.
  \end{equation}
  By introducing the Lagrangian multiplier $u$ to the constraint $ A x-y=z$, we have the Lagrangian function
  \begin{equation}\label{new48}
     L\left(x,y,z,u\right) = f\left(x\right)+h\left(y\right)+\left(\lambda^{k}\right) ^{T }z+\beta^{\frac{1}{p} }\frac{1}{1+\frac{1}{p} }\left\lVert z\right\rVert ^{1+\frac{1}{p} }+u^{T }\left(A x-y-z\right).
  \end{equation}
  Due to the optimality condition, it holds that
  \begin{equation}\label{new49}
   \left\{ {\begin{array}{*{20}{c}}
     {-A ^{T }u\in \partial f\left(x^{k+1}\right) }\\
     u \in \partial h\left(y^{k+1}\right) \\
     A x^{k+1}-y^{k+1}=z\\
     u = \lambda ^{k}+\beta^{\frac{1}{p} } \partial \left(\frac{1}{1+\frac{1}{p} } \left\lVert z\right\rVert^{1+\frac{1}{p} }\right) 
     \end{array}}. \right.
  \end{equation}
  In view of \cref{newlemma_3}, we have
  \begin{equation}\label{new50}
   \left\{ {\begin{array}{*{20}{c}}
     x^{k+1} \in \partial f^{\ast }\left(-A ^{T }u\right) \\
     y^{k+1} \in \partial h^{\ast }\left(u\right)
     \end{array}}. \right.
  \end{equation}   
  Next, we discuss the two cases of $z\neq 0$ and $z = 0$ respectively.

  Case 1 ($z\neq 0$): According to \cref{new49}, $ u = \lambda ^{k}+\beta^{\frac{1}{p} } z/ \left\lVert z\right\rVert ^{1-\frac{1}{p} }=\lambda ^{k+1}$, which implies that 
  \begin{equation}\label{new51}
       \frac{\left\lVert u-\lambda ^{k}\right\rVert }{\beta^{\frac{1}{p} } } = \left\lVert z\right\rVert ^{\frac{1}{p} }.
  \end{equation}     
  Hence, we have
  \begin{equation}\label{new52}
      z = \frac{1}{\beta } \left\lVert u-\lambda^ {k}\right\rVert ^{p-1}\left(u-\lambda^ {k}\right).
  \end{equation}      
  Combining \cref{new49}, \cref{new50}, and \cref{new52}, it holds that
  \begin{equation}\label{new53}
    0 \in -A \partial f^{\ast }\left(-A ^{T }u\right)+ \partial h^{\ast }\left(u\right)+\frac{1}{\beta } \left\lVert u-\lambda ^{k}\right\rVert ^{p-1}\left(u-\lambda ^{k}\right).
  \end{equation}   
  And this is the optimality condition of \cref{new46}.Thus, \cref{new46} holds.

  Case 2 ($z = 0$): According to \cref{new49}, it can be easily verified that 
  \begin{equation}\label{new54}
   \left\{ {\begin{aligned}
     A x^{k+1} -y^{k+1}&=0\\
     u &= \lambda ^{k}
     \end{aligned}}. \right.
  \end{equation}   
  Similarly, we have
  \begin{equation}\label{new55}
   0 \in -A \partial f^{\ast }\left(-A ^{T }u\right)+ \partial h^{\ast }\left(u\right)+\frac{1}{\beta } \left\lVert u-\lambda ^{k}\right\rVert ^{p-1}\left(u-\lambda ^{k}\right).
 \end{equation}  
 Therefore, \cref{new46} holds.
\end{proof}

\begin{remark}
  Note that \cref{new46} in \cref{newTheorem3} is the $p$th-order proximal point method for the dual problem. On the other hand, if we have 
  \begin{equation}\label{new56}
   u \in \arg \min \left\{f^{\ast }\left(-A^T u\right) +h^{\ast }\left(u\right)+\frac{1}{\beta} \frac{1}{p+1} \left\lVert u-\lambda ^{k+1}\right\rVert^{p+1} \right\}, \nonumber
  \end{equation}
  the variable $ \left(x^{k+1},y^{k+1}\right) $ in the $p$th-order ALM can be recovered by taking $ x^{k+1}\in \partial f^{\ast }\left(-A ^{T }u\right) $ and $ y^{k+1}\in \partial h^{\ast }\left(u\right) $. Therefore, the $p$th-order ALM for the primal problem is equivalent to the $p$th-order proximal point method for the dual problem. 
\end{remark}

\subsection{The subproblem in the pth-order ALM}
The update for $x^{k+1}$ is the essential step in the $p$th-order ALM. In general, $x^{k+1}$  has no closed-form solution, and an iterative algorithm is needed to solve it. Denote 
\begin{equation}\label{nn1}
  \psi (x) = ( \lambda ^{k} ) ^T (A x-b)+\frac{\beta^{\frac{1}{p} } }{1+\frac{1}{p} }\left\lVert A x-b\right\rVert ^{1+\frac{1}{p} },
\end{equation}
 and the update for $x^{k+1}$ can be simplified as 
\begin{equation}\label{n1}
  x^{k+1} \in \arg \min  \left\{ \psi (x)+f\left(x\right)  \vert x \in  \mathcal{X} \right\}.
\end{equation}
The following lemma demonstrates that the gradient of $\psi (x)$ is Hölder Lipschitz continuous.
\begin{lemma}\label{newlemma_4}
  For any $p \geq 1$, $\nabla \psi (x)$ is  Hölder Lipschitz continuous, i.e.
  \begin{equation}\label{n2}
    \left\lVert \nabla \psi (y)-\nabla \psi (x)\right\rVert \leq M_p \left\lVert y-x\right\rVert ^{\frac{1}{p} }, \quad \forall x,y \in \mathbb{R} ^{n},
  \end{equation}
  where $M_p = \left[\left(p+1\right)2^{p-2} \right]^{\frac{1}{p}} \beta ^{\frac{1}{p} }\left\lVert A\right\rVert ^{1+\frac{1}{p} }$.
\end{lemma}
\begin{proof}
  $\nabla \psi (x) = A^{T}\lambda ^{k}+\beta^{\frac{1}{p} }A^{T}i_p(Ax-b)$. Here, we only give the proof of the case that $\nabla \psi (x)\neq 0$ and $\nabla \psi (y)\neq 0$, 
  and the proofs for the other cases are similar.
  \begin{equation}\label{n3}
    \begin{aligned}
    \left\lVert \nabla \psi (x)-\nabla \psi (y)\right\rVert  &= \left\lVert \beta ^{\frac{1}{p} }A^{T}\left(\frac{Ax-b}{\left\lVert Ax-b\right\rVert ^{1-\frac{1}{p} }} -\frac{Ay-b}{\left\lVert Ay-b\right\rVert ^{1-\frac{1}{p} }}\right) \right\rVert  \\
    & \leq \beta ^{\frac{1}{p} }\left\lVert A\right\rVert \left\lVert \frac{Ax-b}{\left\lVert Ax-b\right\rVert ^{1-\frac{1}{p} }} -\frac{Ay-b}{\left\lVert Ay-b\right\rVert ^{1-\frac{1}{p} }} \right\rVert.
    \end{aligned}
   \end{equation}
   Due to the the uniformly convexity of the function $\frac{1}{p+1} \left\lVert x\right\rVert^{p+1} $, the following inequality holds 
   \begin{equation}\label{n4}
    \begin{aligned}
      \frac{2}{p+1} \left(\frac{1}{2} \right)^{p-1} \left\lVert u-v\right\rVert^{p+1}&\leq \left(u-v\right)^{T} (\left\lVert u\right\rVert^{p-1}u-\left\lVert v\right\rVert^{p-1}v )  \\
    & \leq \left\lVert u-v\right\rVert \left\lVert \left\lVert u\right\rVert^{p-1}u-\left\lVert v\right\rVert^{p-1}v\right\rVert .
    \end{aligned}
   \end{equation}
  Taking $u = \frac{Ax-b}{\left\lVert Ax-b\right\rVert^{1-\frac{1}{p} } } $ and $v = \frac{Ay-b}{\left\lVert Ay-b\right\rVert^{1-\frac{1}{p} } } $ in \cref{n4}, we have
  \begin{equation}\label{n5}
    \begin{aligned}
      \left\lVert \frac{Ax-b}{\left\lVert Ax-b\right\rVert ^{1-\frac{1}{p} }} -\frac{Ay-b}{\left\lVert Ay-b\right\rVert ^{1-\frac{1}{p} }} \right\rVert  &\leq \left[\left(p+1\right)2^{p-2} \right]^{\frac{1}{p}} \left\lVert Ax-Ay\right\rVert ^{\frac{1}{p} }  \\
    & \leq \left[\left(p+1\right)2^{p-2} \right]^{\frac{1}{p}} \left\lVert A\right\rVert^{\frac{1}{p} } \left\lVert x-y\right\rVert ^{\frac{1}{p} }.
    \end{aligned}
   \end{equation}
   Combining \cref{n3} and \cref{n5}, we get the final result \cref{n2}.
\end{proof}

Thus, the universal fast gradient \cite{nesterov2015universal} can be used to solve the subproblem \cref{n1}. The upper bound for the number of iterations required to obtain $\epsilon $-solution of \cref{n1} ($\psi(x^{k})+f(x^{k})-\psi(x^{\ast})-f(x^{\ast}) \leq \epsilon$) is 
\begin{equation}\label{n6}
  k \leq \left(\frac{2^{\frac{3p+5}{2p} }M_p}{\epsilon } \right)^{\frac{2p}{p+3} } R^{\frac{p+1}{p+3} },
\end{equation}
where $M_p$ is defined in \cref{newlemma_4}, and $R$ is the distance from a starting point to the solution. 

\subsection{Stopping  criteria for the pth-order ALM}
In this part, we discuss the stopping criteria under the assumption that $\mathcal{X} = \mathbb{R} ^{n} $ in \cref{new24}. In this case, the optimality conditions for \cref{new24} are primal and dual feasibility
\begin{equation}\label{new57}
    Ax^{\ast}-b=0, \qquad 0 \in \partial f(x^{\ast})+A^{T}\lambda ^{\ast},
\end{equation}
respectively. By the definition of the $p$th-order ALM \cref{new34},
\begin{equation}\label{new58}
  \begin{aligned}
   0&\in \partial f(x^{k+1})+A^{T}(\lambda ^{k}+\beta ^{\frac{1}{p} }i_p(Ax^{k+1}-b)) \\ &= \partial f(x^{k+1})+A^{T}\lambda ^{k+1} .
  \end{aligned}
\end{equation} 
This implies that the iterate $(x^{k+1},\lambda ^{k+1})$ generated by the $p$th-order ALM always satisfies the dual feasibility. Thus, the convergence of the primal residual $Ax^{k+1}-b$ yields the optimality. Define $r^{k}$ as
\begin{equation}\label{new59}
    r^{k} = \left\lVert Ax^{k}-b\right\rVert,
\end{equation}
and a reasonable stopping criterion is that $r^{k}$ is sufficiently small, 
\begin{equation}\label{new60}
  r^{k} \leq \varepsilon ,
\end{equation}
where $\varepsilon$ is the  tolerances for the primal  feasibility.

\section{Numerical experiments}\label{sec6}
In this section, we demonstrate the performance of the $p$th-order ALM with some options for  $p$ by applying it to  some common test problems. All  experiments are implemented in MATLAB R2020b on a computer of 16 GB RAM, AMD Core R7-4800H CPU, @ 2.90GHz.

 Here, we adopt the universal fast gradient method to solve the x-subproblem approximately. Denote $G(x)= x-\mathsf{P}\mathrm{rox}_{f}\left(x- \nabla \psi (x) \right)$, where $\psi (x)$ is defined in \cref{nn1}. Obviously, when $G(z^{k})=0$, $z^{k}$ is the solution of the x-subproblem. Therefore, the termination criteria for the universal fast gradient method can be 
\begin{equation}\label{new64}
  \left\lVert G(z^{k}) \right\rVert \leq \varepsilon _{sub} ,
\end{equation}
where $\varepsilon _{sub}$ controls the accuracy of the solution to the x-subproblem.

\subsection{Basis pursuit}
Basis pursuit (BP) problem \cite{chen2001atomic, mota2011distributed} is a basic problem in the theory of compressed sensing, and it aims to  find a sparse solution to an underdetermined system of linear equations. The BP problem can be expressed as 
\begin{equation}\label{exper_1}
  \min \left\{\left\lVert x\right\rVert _1  | \; Ax=b\right\} ,
\end{equation} 
where $A \in \mathbb{R}^{m\times n}$, $b \in \mathbb{R}^{m}$ with  $m<n$. Here, we randomly generate $A$, where $A_{ij}$ follows the  normal distribution $ \mathcal{N} (0,1)$. To generate $b$, we first  create a $n\times 1$ sparse random vector $u_0$ with the density $0.2$, and then let $b=Au_0$. In this experiment, we take $m=100$ and $n = 500$. 

\begin{figure}[tbhp]
  \centering
  \subfloat[$\beta = 0.5$]{\label{fig:a}\includegraphics[scale=0.14]{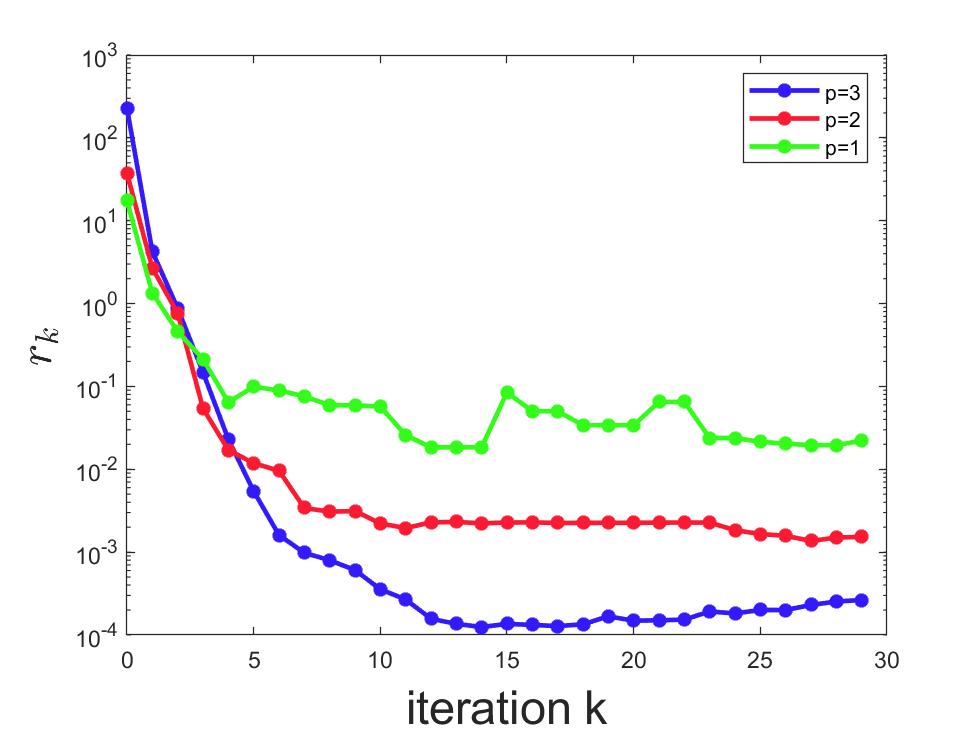}}
  \subfloat[$\beta = 2$]{\label{fig:b}\includegraphics[scale=0.14]{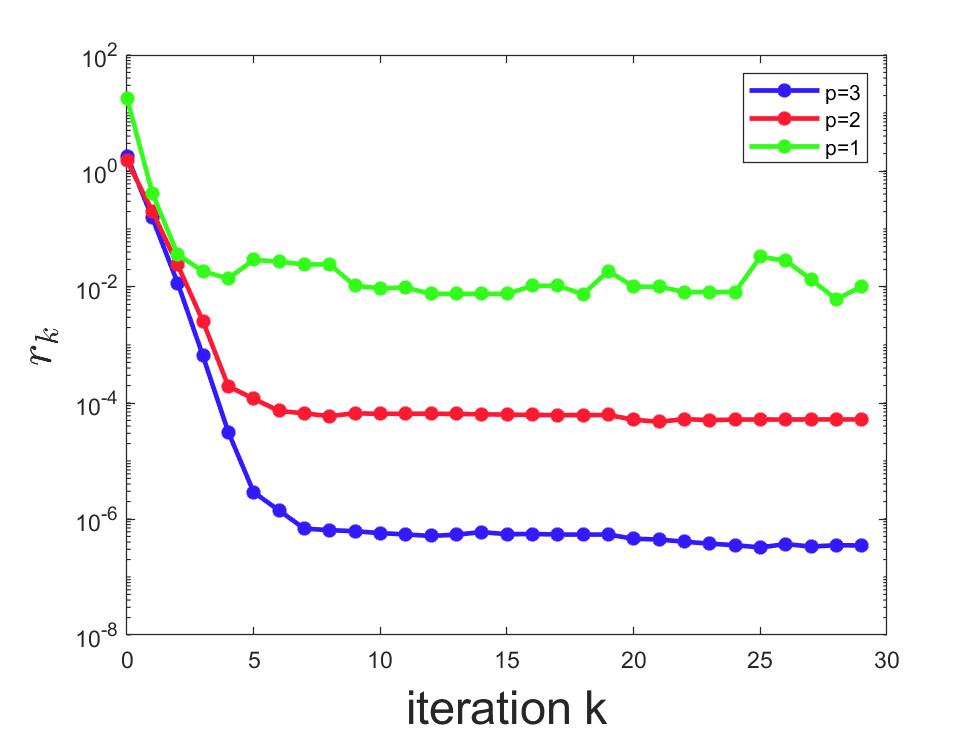}}\\
  \subfloat[$\beta = 5$]{\label{fig:c}\includegraphics[scale=0.14]{1.jpg}}
  \subfloat[$\beta = 10$]{\label{fig:d}\includegraphics[scale=0.14]{2.jpg}}
  \caption{The curve of  norms of primal residual $r_k$ with different $\beta$ in the BP experiment.}
  \label{fig1}
\end{figure}
\begin{figure}[tbhp]
  \centering
  \subfloat[$\varepsilon _{sub}= 0.1$]{\label{fig:a1}\includegraphics[scale=0.14]{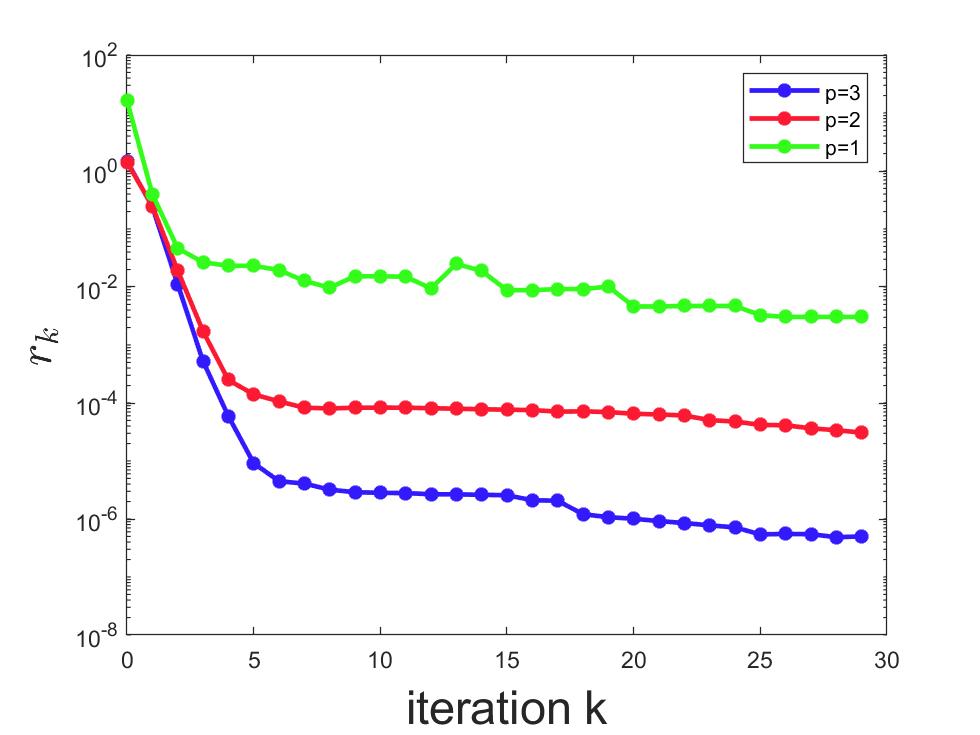}}
  \subfloat[$\varepsilon _{sub} = 0.05$]{\label{fig:b1}\includegraphics[scale=0.14]{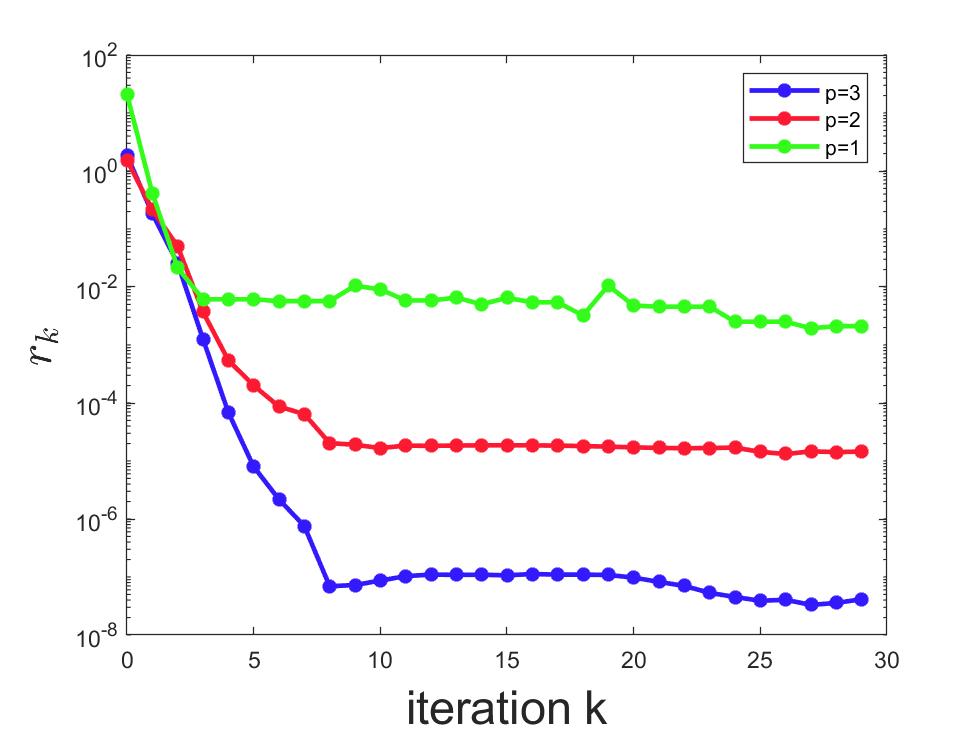}}\\
  \subfloat[$\varepsilon _{sub} = 0.01$]{\label{fig:c1}\includegraphics[scale=0.14]{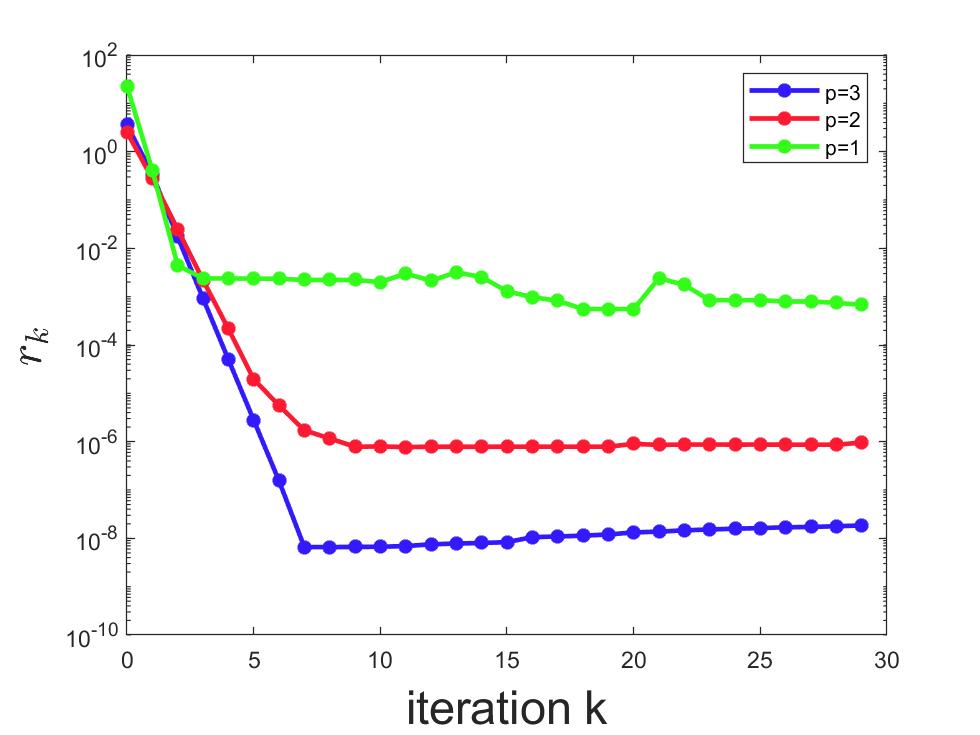}}
  \subfloat[$\varepsilon _{sub} = 0.005$]{\label{fig:d1}\includegraphics[scale=0.14]{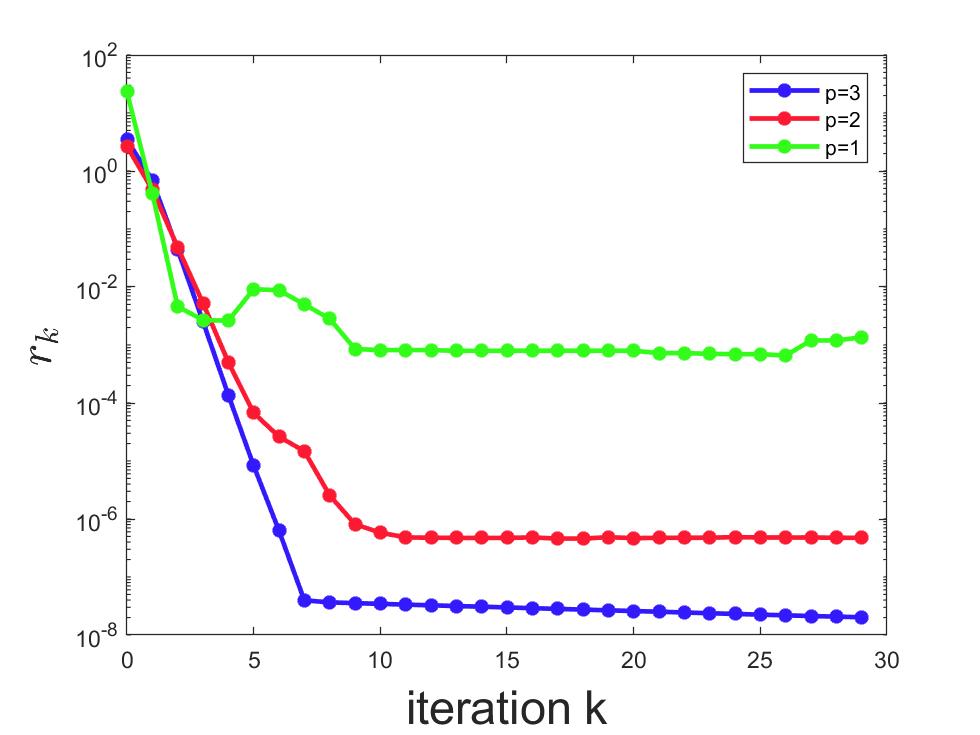}}
  \caption{The curve of  norms of primal residual $r_k$ with different $\varepsilon _{sub}$.}
  \label{fig2}
\end{figure}
We apply the $p$th-order ALM to the BP  problem with $f(x)=\left\lVert x\right\rVert _1 $. In this  experiment, we first fix $\varepsilon _{sub}=10^{-1}$ to illustrate the the numerical performance of $p$th-order ALM ($p=1,2,3$) with some different $\beta$.
\cref{fig1} shows the convergence curve for the BP problem. Obviously, with some different  $\beta$ ($\beta=0.5,2,5,10$), the 2nd-order and 3rd-order ALM converge much faster than the original ALM ($p=1$). Next, we fix $\beta = 2$ to illustrate the the numerical performance of $p$th-order ALM ($p=1,2,3$) with some different $\varepsilon _{sub}$. \cref{fig2} shows the convergence curve for the BP problem with some different accuracy of the solution to the x-subproblem. Similarly, in this setting, the 2nd-order and 3rd-order ALM converge much faster than the original ALM ($p=1$).

\begin{figure}[tbhp]
  \centering
  \subfloat[$\varepsilon _{sub}= 0.1$]{\label{fig:a2}\includegraphics[scale=0.18]{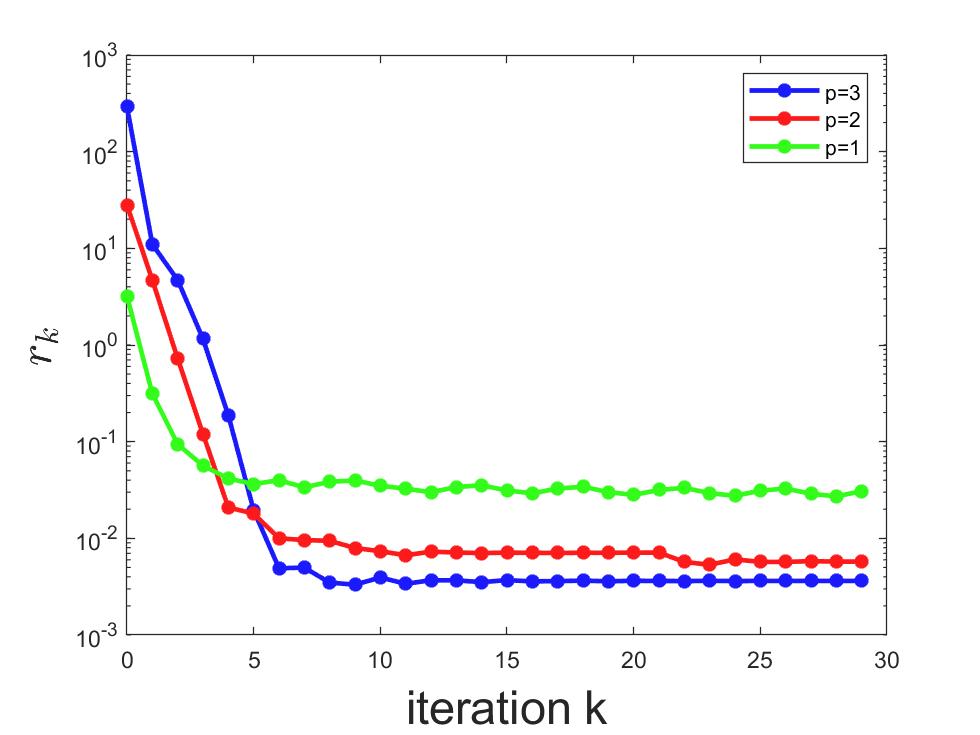}}
  \subfloat[$\varepsilon _{sub} = 0.05$]{\label{fig:b2}\includegraphics[scale=0.18]{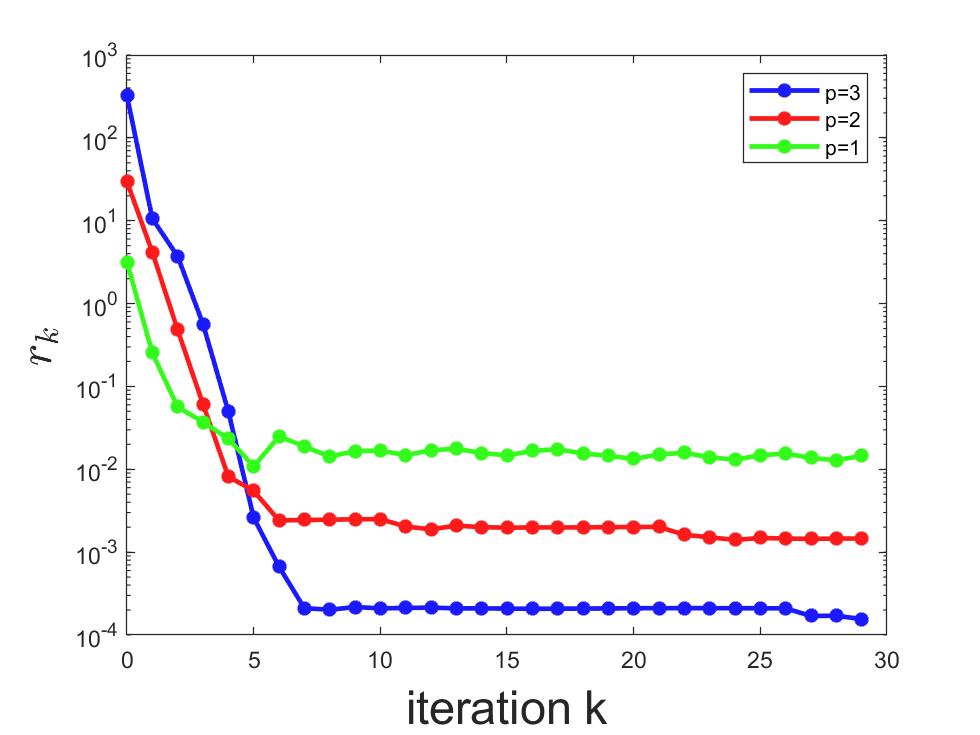}}\\
  \subfloat[$\varepsilon _{sub} = 0.01$]{\label{fig:c2}\includegraphics[scale=0.18]{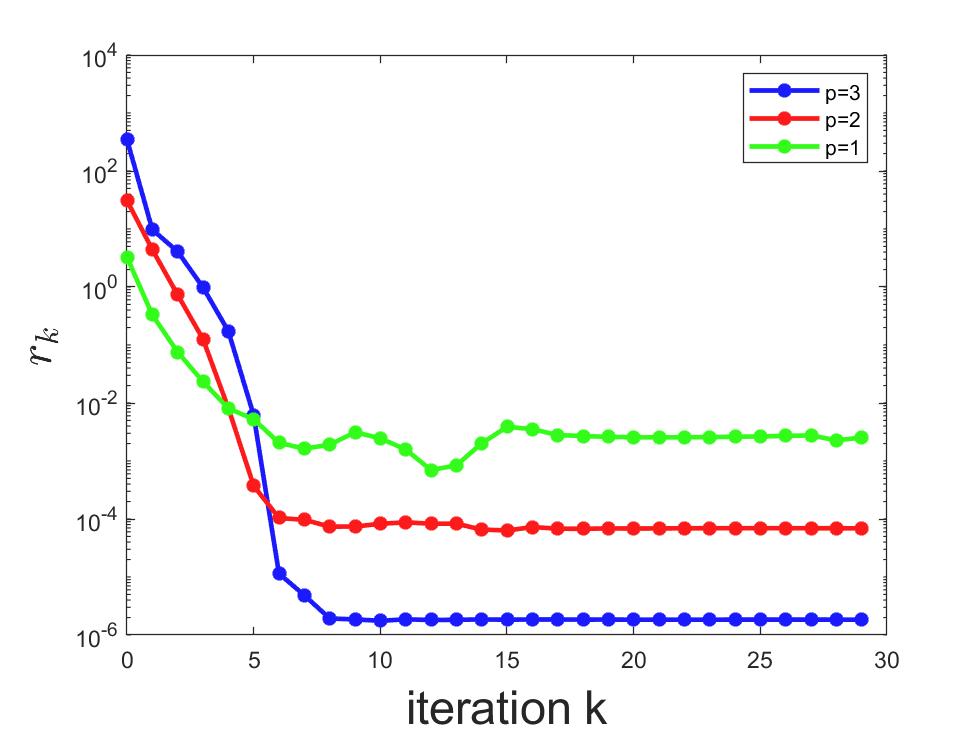}}
  \subfloat[$\varepsilon _{sub} = 0.005$]{\label{fig:d2}\includegraphics[scale=0.18]{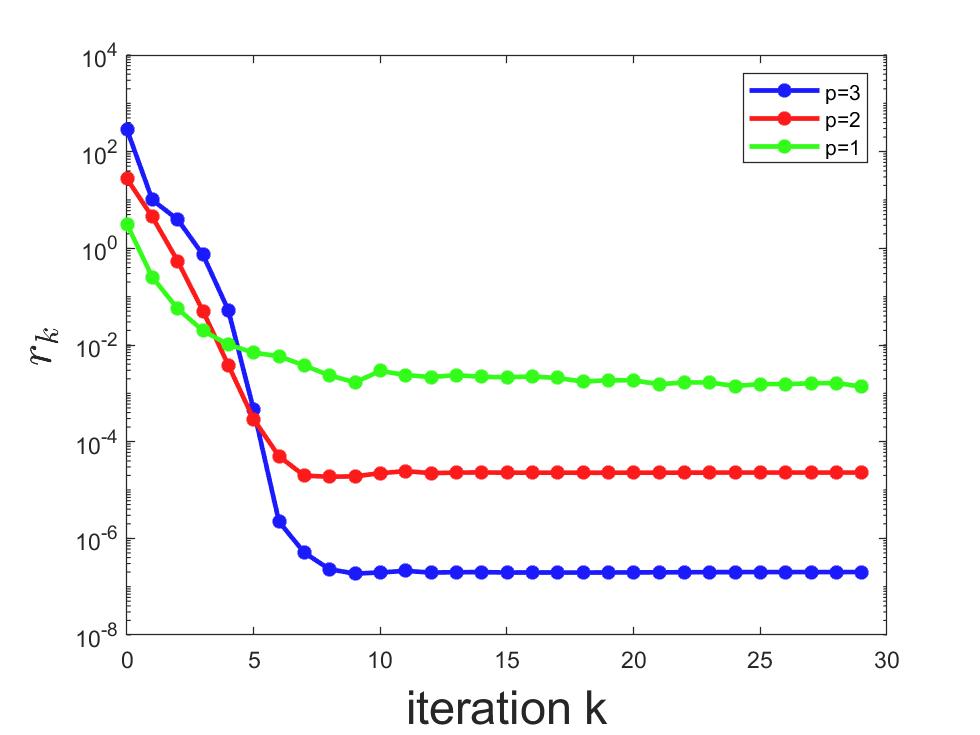}}
  \caption{The curve of  norms of primal residual $r_k$ with different $\varepsilon _{sub}$ in the MC experiment.}
  \label{fig3}
\end{figure}
\subsection{Matrix completion} Matrix completion (MC) problems \cite{recht2011simpler,candes2010matrix} are widely used in recommendation systems and image processing. A common model of the MC problem is the following optimization \cite{cai2010singular}
\begin{equation}\label{exper_113}
  \min \left\{\left\lVert X\right\rVert _{\ast }  | \; X_{ij}=M_{ij}, \; (ij)\in \Lambda \right\} ,
\end{equation} 
where $\left\lVert \cdot \right\rVert _{\ast }$ is the nuclear norm of a given matrix, $M\in \mathbf{R}^{m\times n}$ and $\Lambda$ is the elements indices set of $M$. In this experiment, we set $m=n=50$, and generate  a $50\times 50$ sparse random matrix $M$ with the density $0.1$, and the parameter $\beta $ in $p$th-order ALM is set as 5.

\cref{fig3} shows the convergence curve for the MC problem with some different accuracy of the solution to the x-subproblem. The result illustrates that high-order ALM ($p>1$) can  quickly obtain the high-precision solution to the primal problem with relatively low accuracy of the solution to the x-subproblem.

\section{conclusion}
In this paper, we present the $p$th-order PPA for the VI problem. This can be view as the generalization of the high-order proximal-point iteration proposed by Nesterov \cite{nesterov2021inexact2}. Based on the $p$th-order PPA, we also present the $p$th-order ALM for the linearly constrained convex optimization problem.  Both the $p$th-order PPA and the $p$th-order ALM have the convergence rate $O \left(1/k^{p/2}\right)$.

When the  linearly constrained convex optimization problem is separable, ADMM is an efficient method to solve it. The ADMM can be viewed as the splitting  version of the ALM. In this study, we have developed the $p$th-order ALM. Thus, how to extend ADMM to the high-order case is worth studying.

\bibliographystyle{siamplain}
\bibliography{references}
\end{document}

%% file: ex_shared.tex
\usepackage{graphicx,epstopdf} 
\usepackage[caption=false]{subfig} 

\usepackage{lipsum}
\usepackage{amsfonts}
\usepackage{graphicx}
\usepackage{epstopdf}
\usepackage{algorithmic}
\ifpdf
  \DeclareGraphicsExtensions{.eps,.pdf,.png,.jpg}
\else
  \DeclareGraphicsExtensions{.eps}
\fi

\newsiamremark{remark}{Remark}
\newsiamremark{hypothesis}{Hypothesis}
\crefname{hypothesis}{Hypothesis}{Hypotheses}
\newsiamthm{claim}{Claim}

\headers{pth-order PPA for monotone VI problem}{Jingyu Gao and Xiurui Geng}

\title{High-order proximal point algorithm  for the monotone variational inequality problem and its application\thanks{Submitted to the editors
DATE.
}}

\author{Jingyu Gao\thanks{Aerospace Information Research Institute, Chinese    Academy of Sciences; School of Electronic, Electrical and Communication Engineering, University of Chinese Academy of Sciences; Key Laboratory of Technology in Geo-Spatial Information Processing and Application System, Chinese Academy of Sciences 
  .}
\and Xiurui Geng\thanks{Aerospace Information Research Institute, Chinese    Academy of Sciences; School of Electronic, Electrical and Communication Engineering, University of Chinese Academy of Sciences; Key Laboratory of Technology in Geo-Spatial Information Processing and Application System, Chinese Academy of Sciences 
  (\email{genxr@sina.com.cn}).}
}

\usepackage{amsopn}


%% file: ex_article.bbl
\begin{thebibliography}{10}

\bibitem{boyd2011distributed}
{\sc S.~Boyd, N.~Parikh, E.~Chu, B.~Peleato, J.~Eckstein, et~al.}, {\em
  Distributed optimization and statistical learning via the alternating
  direction method of multipliers}, Foundations and Trends{\textregistered} in
  Machine learning, 3 (2011), pp.~1--122.

\bibitem{boyd2004convex}
{\sc S.~P. Boyd and L.~Vandenberghe}, {\em Convex optimization}, Cambridge
  university press, 2004.

\bibitem{cai2010singular}
{\sc J.-F. Cai, E.~J. Cand{\`e}s, and Z.~Shen}, {\em A singular value
  thresholding algorithm for matrix completion}, SIAM Journal on optimization,
  20 (2010), pp.~1956--1982.

\bibitem{candes2010matrix}
{\sc E.~J. Candes and Y.~Plan}, {\em Matrix completion with noise}, Proceedings
  of the IEEE, 98 (2010), pp.~925--936.

\bibitem{chen2001atomic}
{\sc S.~S. Chen, D.~L. Donoho, and M.~A. Saunders}, {\em Atomic decomposition
  by basis pursuit}, SIAM review, 43 (2001), pp.~129--159.

\bibitem{combettes2007douglas}
{\sc P.~L. Combettes and J.-C. Pesquet}, {\em A douglas--rachford splitting
  approach to nonsmooth convex variational signal recovery}, IEEE Journal of
  Selected Topics in Signal Processing, 1 (2007), pp.~564--574.

\bibitem{gol1979modified}
{\sc E.~G. Gol’shtein and N.~Tret’yakov}, {\em Modified lagrangians in
  convex programming and their generalizations}, Point-to-Set Maps and
  Mathematical Programming,  (1979), pp.~86--97.

\bibitem{he20121}
{\sc B.~He and X.~Yuan}, {\em On the o(1/n) convergence rate of the
  douglas--rachford alternating direction method}, SIAM Journal on Numerical
  Analysis, 50 (2012), pp.~700--709.

\bibitem{he2015non}
{\sc B.~He and X.~Yuan}, {\em On non-ergodic convergence rate of
  douglas--rachford alternating direction method of multipliers}, Numerische
  Mathematik, 130 (2015), pp.~567--577.

\bibitem{hestenes1969multiplier}
{\sc M.~R. Hestenes}, {\em Multiplier and gradient methods}, Journal of
  optimization theory and applications, 4 (1969), pp.~303--320.

\bibitem{hong2017linear}
{\sc M.~Hong and Z.-Q. Luo}, {\em On the linear convergence of the alternating
  direction method of multipliers}, Mathematical Programming, 162 (2017),
  pp.~165--199.

\bibitem{lions1979splitting}
{\sc P.-L. Lions and B.~Mercier}, {\em Splitting algorithms for the sum of two
  nonlinear operators}, SIAM Journal on Numerical Analysis, 16 (1979),
  pp.~964--979.

\bibitem{mota2011distributed}
{\sc J.~F. Mota, J.~M. Xavier, P.~M. Aguiar, and M.~Puschel}, {\em Distributed
  basis pursuit}, IEEE Transactions on Signal Processing, 60 (2011),
  pp.~1942--1956.

\bibitem{nemirovski2004prox}
{\sc A.~Nemirovski}, {\em Prox-method with rate of convergence o (1/t) for
  variational inequalities with lipschitz continuous monotone operators and
  smooth convex-concave saddle point problems}, SIAM Journal on Optimization,
  15 (2004), pp.~229--251.

\bibitem{nesterov2015universal}
{\sc Y.~Nesterov}, {\em Universal gradient methods for convex optimization
  problems}, Mathematical Programming, 152 (2015), pp.~381--404.

\bibitem{nesterov2021inexact2}
{\sc Y.~Nesterov}, {\em Inexact accelerated high-order proximal-point methods},
  Mathematical Programming,  (2021), pp.~1--26.

\bibitem{nesterov2021inexact}
{\sc Y.~Nesterov}, {\em Inexact high-order proximal-point methods with
  auxiliary search procedure}, SIAM Journal on Optimization, 31 (2021),
  pp.~2807--2828.

\bibitem{powell1969method}
{\sc M.~J. Powell}, {\em A method for nonlinear constraints in minimization
  problems}, Optimization,  (1969), pp.~283--298.

\bibitem{recht2011simpler}
{\sc B.~Recht}, {\em A simpler approach to matrix completion.}, Journal of
  Machine Learning Research, 12 (2011).

\bibitem{rockafellar1976augmented}
{\sc R.~T. Rockafellar}, {\em Augmented lagrangians and applications of the
  proximal point algorithm in convex programming}, Mathematics of operations
  research, 1 (1976), pp.~97--116.

\end{thebibliography}
